\documentclass{amsart}

\usepackage{amsfonts}
\usepackage{amssymb}
\usepackage{amstext}

\usepackage[]{amscd}
\newtheorem{theo}{Theorem}
\newtheorem{lem}{Lemma}
\newtheorem{cor}{Corollary}

\newtheorem{rem}{Remark}
\begin{document}

\title{The Mahler measure and the L-series of a singular $K3$ surface}

\author{Marie Jos\'e Bertin}


\email{bertin@math.jussieu.fr}
\date{\today}

\curraddr{Universit\'e Pierre et Marie Curie (Paris 6), Institut
de Math\'ematiques, 175 rue du Chevaleret, 75013 PARIS, France}

\begin{abstract}
We present the first example of a polynomial $P_{10}
$ defining a singular $K3$ surface $Y_{10}$ whose Mahler measure is expressed in terms of the Mahler measure of the faces of its Newton polyhedron and the $L$-series of the $K3$-surface. 
\end{abstract}

\maketitle
\section{Introduction}
The logarithmic Mahler measure $m(P)$ of a Laurent polynomial $P\in \mathbb{C}[X_1^{\pm 1} ,...,X_n^{\pm 1} ]$ is defined by

$$
m(P):=\frac {1}{(2\pi i)^n} \int_{\mathbb{T} ^n}\log \vert P(x_1,...,x_n)\vert\frac{dx_1}{x_1}...\frac{dx_n}{x_n}
$$
where $\mathbb{T}^n$ is the $n$-torus $\{(x_1,...,x_n)\in \mathbb{C}^n /
\vert x_1\vert =...=\vert x_n\vert =1\}$

and has a strange link with Calabi-Yau varieties. Its story can be explained briefly considering the Mahler measure of polynomials
$$x_0+x_1+x_2+\hdots +x_n.$$
The known results are $m(x_0+x_1)=0$, $m(x_0+x_1+x_2)=L'(\chi_{-3},-1)$ (Smyth)\cite{Sm}, $m(x_0+x_1+x_2+x_3)=\frac {7}{2\pi^2}\zeta(3)$ (Smyth)\cite{Sm}, $m(x_0+x_1+x_2+x_3+x_4)$ conjectured to be related to $L'(f,-1)$ where $f$ is a cusp form of weight $3$ and conductor $15$ (Rodriguez-Villegas) \cite{Boy}. The cusp form $f$ is equal to
$$f=\sum_{m,n\in \mathbb Z}q^{m^2+mn+4n^2}\eta(q)\eta(q^3)\eta(q^5)\eta(q^{15}).$$
 
This conjecture can be verified to a high accuracy \cite{RVTV} and $L(f,s)$ is equal to the $L$-series of the $K3$-surface \cite{PTV} which is the minimal resolution of singularities of the surface in $\mathbb P^4$ given by the equations
$$x_0+x_1+x_2+x_3+x_4=0$$
$$\frac {1}{x_0}+\frac {1}{x_1}+\frac {1}{x_2}+\frac {1}{x_3}+\frac {1}{x_4}=0.$$ 
Finally, $m(x_0+x_1+x_2+x_3+x_4+x_5)$ is conjectured to be related to $L'(g,-1)$ where $g$ is a cusp form of weight $4$ and conductor $6$ (Rodriguez-Villegas) \cite{Boy}. It is also verified to a high accuracy and the $L$-series is equal to the $L$-series of the Barth-Nieto quintic \cite{HVS}. 
\bigskip

In fact, these two last guesses of Rodriguez-Villegas have been made possible thanks to a deep and intriguing insight of Maillot \cite{Ma} concerning the Mahler measure of non-reciprocal polynomials i.e. polynomials  $P(x_1, \hdots ,x_n)$ such that
$P(1/x_1,\hdots , 1/x_n)/P(x_1, \hdots ,x_n)$ is not a monomial. The logarithmic Mahler measure of a polynomial can be interpreted as the integration of a differential form on a variety and when the polynomial is non-reciprocal, the variety in question is an algebraic variety of dimension less than $n$. Moreover, the expression of the Mahler measure should be encoded in the cohomology of this variety. In the previous examples you may observe that the variety is a curve of genus $0$ if $n=2$, a product of three planes if $n=3$, a $K3$-surface if $n=4$ and the Barth-Nieto quintic if $n=5$.
\bigskip

On the opposite side, when the polynomial is reciprocal, Deninger\cite{D} is the pioneer and my result here concerns generalizations of Deninger's result. Note that Rodriguez-Villegas's conjectures require results for reciprocal polynomials. That is for me a strong motivation to compute many examples in that direction.

\bigskip
Let $P(x,y)$ be a polynomial in two variables with integer coefficients. Suppose that $P$ does not vanish on 2-torus $\mathbb T^2$ with $\mathbb T^2 =\{(x,y)\in \mathbb C^2 \,\, / \mid x \mid =\mid y \mid =1 \}$. If $P$ is associated to an elliptic curve $E$, such that the polynomials of the faces $P_F$ of $P$ (defined in terms of the Newton polygon of $P$) are cyclotomic, the following relation between the logarithmic Mahler measure of $P$, $m(P)$, and the $L$-series of the elliptic curve $E$
$$m(P)\doteq \frac {N}{4\pi ^2}L(E,2)=L'(E,0)$$
is conjectured to hold. ( $\doteq $ means equality up to a rational coefficient and $N$ is the conductor of $E$ \cite{D} \cite {Bo} \cite {RV1} \cite {RV2}).
After Deninger's guess of such explicit formulae \cite{D}, Boyd gave a lot of examples of the same type where the rational coefficient was determined numerically \cite{Bo}. Then Rodriguez-Villegas proved some formulae when the corresponding elliptic curve $E$ has complex multiplication \cite{RV1} \cite{RV2}.
His proof uses the expression of the logarithmic Mahler measure of a polynomial in terms of Eisenstein-Kronecker series. For more examples see \cite {Ber1}.
Recently Brunault \cite{Br}, Lalin and Rogers \cite{L-R}  gave also some proofs of such relations.

In a previous paper \cite{Ber3} we proved the analog of these formulae for three particular $P$ of the same family $(P_k)$
$$P_k=x^2yz+xy^2z+xyz^2+t^2(xy+xz+yz)-kxyzt.$$
These particular $P$ define singular modular $K3$ hypersurfaces and their Mahler measure is only related to the $L$-series of the corresponding $K3$-surface. 

This paper deals with the polynomial
$$P_{10}=x^2yz+xy^2z+xyz^2+t^2(xy+xz+yz)-10xyzt,$$
defining the $K3$-hypersurface $Y_{10}$.
This polynomial belongs to the first family $P_k$ of $K3$ hypersurfaces whose logarithmic Mahler measures have been studied in \cite {Ber2}.  As explained in \cite {Ber2}, the derivative of the logarithmic Mahler measure with respect to the parameter is a period of the corresponding $K3$ hypersurface; thus, only the transcendental lattice is relevant for the Mahler measure. The generic member of the family has Picard number $19$. Special members of the family have Picard number $20$, thus a $2$-dimensional transcendental lattice. The polynomial $P_{10}$ is among these special members. The $L$-series of $Y_{10}$ corresponds to just a $2$-dimensional piece of the $22$-dimensional $H^2$ of the $K3$ hypersurface. Since this $K3$ hypersurface has Picard number $20$, this $2$-dimensional piece is its transcendental lattice. The difference with the previous examples is the following. For computing the determinant of the Picard lattice, we need a more elaborate desingularization of the singular fibers. For that purpose, we use the  N\'eron's model \cite{N}. But the main difficulty was to find an infinite section. This section is defined over $\mathbb Q(\sqrt{-3})$ and was discovered by Lecacheux \cite{Le}.
Moreover, the formula giving the Mahler measure of $P_{10}$ is more tricky and we use an idea due to Zagier \cite{Za} and explained in \cite{Ber2}. This idea led me to discover how to separate in the expression of the Mahler measure, the contribution of the face and the modular part. We are also indebted to Boyd who guessed numerically the following relation between the Mahler measure of two polynomials of the family, I mean $P_2$ and $P_{10}$
$$m(P_{10})\overset {?}{=} 2d_3+3m(P_2).$$
In fact Boyd remarked that the two moduli $\tau_{10}$ and $\tau_2$ (defined below in section 2) of the corresponding polynomials $P_{10}$ and $P_2$ belong to the same imaginary quadratic field $\mathbb Q(\sqrt{-2})$. So he suspected a relation between the corresponding Mahler measures.

The aim of the paper is to prove the following theorem.

\begin{theo}
Let $Y_{10}$ the $K3$ hypersurface associated to the polynomial $P_{10}$

$$P_{10}=x^2yz+xy^2z+xyz^2+t^2(xy+xz+yz)-10xyzt.$$

If $L(Y_{10},s)$ denotes the $L$-series of the hypersurface $Y_{10}$, one gets the following relations.

1) 
$$L(Y_{10},3)=\frac {1}{2} \sum'_{k,m} \frac {k^2-2m^2}{(k^2+2m^2)^3}=L(f,3),$$
where $f$ is the $CM$-newform of weight $3$ and level $8$ given in [\cite{Sch1}, Table 1],
$$f=q-2q^2-2q^3+4q^4+4q^6-8q^8-5q^9+14q^{11}-8q^{12}+16q^{16}+2q^{17}+10q^{18}-34q^{19}+\cdots .$$

2) Let define as in \cite{Ber3}
$$d_3:=\frac {3\sqrt{3}}{4\pi}L(\chi_{-3},2)=\frac {2\sqrt{3}}{\pi^3}\sum'_{m,k} \frac {1}{(m^2+3k^2)^2}.$$
The Mahler measure and the $L$-series satisfy the equality

$$m(P_{10})=2d_3+\frac {1}{9}\frac {\left |\det T_{Y_{10}} \right |^{3/2}}{\pi^3} L(Y_{10},3),$$
where $T_{Y_{10}}$ denotes the transcendental lattice of the surface $Y_{10}$.
\end{theo}

\begin{rem}
The $L$-series of the two K3-surfaces $Y_2$ and $Y_{10}$ corresponding respectively to the polynomials $P_2$ and $P_{10}$ are the same. So, by Tate's conjecture there would be an algebraic correspondence between $Y_2$ and $Y_{10}$ \cite{Yu}.
\end{rem}

\section{Previous results}

Let $P_k$ denote the Laurent polynomials

$$
P_k=X+\frac {1}{X}+Y+\frac {1}{Y}+Z+\frac {1}{Z}-k
$$

We recall the following theorem  \cite {Ber2}.
\begin{theo}
Let us write $k=t+\frac {1}{t}$ for 
$$
t=(\frac {\eta (\tau) \eta (6\tau)}{\eta
(2\tau)\eta(3\tau)})^6,$$
where $\eta$ is the Dedekind eta function

$$\eta (\tau)=e^{\frac {\pi i \tau}{12}} \prod_{n\geq 1}(1-e^{2\pi i n\tau})$$
and 
$$q=\exp ^{2\pi i \tau}.$$

Then 
$$
\begin{aligned}  m(P_k)=&\frac {\Im \tau}{8 \pi ^3} \sum'_{m,n}-4\left( 
2\Re \frac {1}{(m\tau+n)^3(m\bar{\tau}+n)}+\frac {1}{(m\tau+n)^2(m\bar{\tau}+n)^2}\right) \\
 &+16\left( 2\Re \frac {1}{(2m\tau+n)^3(2m\bar{\tau}+n)}+\frac {1}{(2m\tau+n)^2(2m\bar{\tau}+n)^2}\right) \\
 &-36\left( 2\Re \frac {1}{(3m\tau+n)^3(3m\bar{\tau}+n)}+\frac {1}{(3m\tau+n)^2(3m\bar{\tau}+n)^2}\right) \\
 &+144\left( 2\Re \frac {1}{(6m\tau+n)^3(6m\bar{\tau}+n)}+\frac {1}{(6m\tau+n)^2(6m\bar{\tau}+n)^2}\right). 
 \end{aligned}
 $$
\end{theo}

\section{The N\'eron's model}
We recall the N\'eron's desingularization only for semi-stable singular fibers. For the other cases we refer to \cite{N}. From now on, we use N\'eron's notations. The key tool is the following theorem.
\begin{theo}(N\'eron)

Let an elliptic curve defined over $\mathbb C[s]$, given by a Weierstrass model $E_s$ and denote $v$ the $s$-adic valuation. 
Suppose that $E_0$ has a double point with distinct tangents and $v(j(E_s))=-m<0$ ($\Longleftrightarrow E_0$ is singular of type $I_m$ in Kodaira's classification). Then, for every integer $l>m/2$, there exists a Weierstrass model $\mathcal E_s$ deduced from $E_s$ by a transformation of the form
\begin{align}
X &=x+qz \\
Y &=y+ux+rz\\
Z &=z
\end{align}
with $q,\quad r,\quad u\in \mathbb C[s]$. The Weierstrass model $\mathcal E_s$ is given by
$$Y^2Z+\lambda XYZ+\mu YZ^2=X^3+\alpha X^2Z+\beta XZ^2+\gamma Z^3$$
with coefficients satisfying
$$v(\lambda^2+4\alpha)=0,\quad v(\mu)\geq l,\quad v(\beta)\geq l, \quad v(\gamma)=m,\quad v(j(\mathcal E_s))=-m.$$
\end{theo}

\bigskip

A singular fiber of type $I_m$ is composed of non singular rational curves $\Theta_0$, $\Theta_1$, ..., $\Theta_{m-1}$ in such a way that $(\Theta_s.\Theta_t)\leq 1$ (i.e. $\Theta_s$ and $\Theta_t$ have at most one simple intersection point) if $s<t$ and $\Theta_r \cap \Theta_s \cap \Theta_t$ is empty for $r<s<t$. So the divisor of the singular fiber can be written
$$\Theta_0+\Theta_1+...+\Theta_{m-1}$$
with
$$(\Theta_0 . \Theta_1)=(\Theta_1 . \Theta_2)=... =(\Theta_s . \Theta_{s+1})=...(\Theta_{m-2} . \Theta_{m-1})=(\Theta_{m-1} . \Theta_0)=1.$$
\subsection{The N\'eron's model}
In the N\'eron's model, the configuration of the fiber $I_m$ of $\mathcal E_s$ over $s=0$ is performed in the space $\mathbb P_2 \times \mathbb P_2 \times ... \times \mathbb P_2$  ($h$ times) for $h=[m/2]$ if $m=2h$.

A point $(X:Y:Z)$ of $\mathcal E_s$ corresponds in the N\'eron's model to the point $(X:Y:Z^{(1)})\times (X:Y:Z^{(2)})\times ...\times (X:Y:Z^{(h)})$ where $(X:Y:Z^{(i+1)})=(X:Y:sZ^{(i)})$ for $ 0\leq i \leq h-1$.

So, if $(X:Y:Z)$ satisfies
$$Y^2Z+\lambda XYZ+\mu YZ^2=X^3+\alpha X^2Z+\beta XZ^2+\gamma Z^3,$$

since $Z^{(1)}=sZ$, the point $(X:Y:Z^{(1)})$ satisfies
$$Y^2Z^{(1)}+\lambda XYZ^{(1)}+\frac {\mu}{s} Y(Z^{(1)})^2=sX^3+\alpha X^2Z^{(1)}+\frac {\beta}{s} X(Z^{(1)})^2+\frac {\gamma}{s^2} (Z^{(1)})^3.$$
Because of valuations, if $h\geq 1$, $\mu /s$, $\beta /s$, $\gamma /s^2$ still belong to $\mathbb C[s]$.

Now if $h\geq 2$, in the singular fiber $s=0$, we get 

$$
Y^2Z^{(1)}+\lambda_0 XYZ^{(1)}-\alpha _0 X^2Z^{(1)}=Z^{(1)}(Y^2+\lambda_0 XY -\alpha _0 X^2)=0,
$$
that is

$$Z^{(1)}(Y-\nu_0 X)(Y-\bar{\nu_0} X)=0$$
since $\lambda_0+4\alpha_0 \neq 0$.

Thus, a point $(X:Y:Z)$ of $\mathcal E_s$, singular over $s=0$, i.e. giving the point $(0:0:1)$ for $s=0$, corresponds to the point $(X:Y:Z^{(1)})=(X/s:Y/s:Z)$ in the first component of the desingularized fiber. For $s=0$ this last point is either still $(0:0:1)$ or $(x_1:\nu_0x_1:1)$ or $(x_1:\bar{\nu_0}x_1:1)$.  In the two last cases, it has been desingularized either on the $\Theta_{0,1}$ or $\Theta_{0,m-1}$.

In the first case the desingularization must be pursued.

At the last step, if $m=2h$, and if the point $(X:Y:Z)$ singular over $s=0$ is not desingularized at the $h-1$ step, then $(X/s^{h-1}:Y/s^{h-1}:Z)$ gives $(0:0:1)$ for $s=0$; but $(X/s^h:Y/s^h:Z)$ gives, for $s=0$, $(x_1:y_1:z_1)$ with $x_1$, $y_1$ and $z_1$ satisfying the equation of a conic $\tilde{C^0}$
$$Y^2Z+\lambda_0 XYZ=\alpha_0 X^2Z+\gamma_m^0 Z^3.$$
If $m=2h+1$, the last step is like the first one.

Finally, keeping track of all these data and denoting like N\'eron, $c^0=(0:0:1)$, $a_1^0=(1:\nu^0:0)$ and $\bar{a_1^0}=(1:\bar{\nu^0}:0)$, $w^0$, $v^0$, $\bar{v^0}$ and $\tilde{v^0}$ respectively the lines $Z=0$, $Y=\nu^0 X$, $Y=\bar{\nu^0} X$ and the conic $\tilde{C^0}$, we can describe the components $\Theta_{0,i}$.

$$
\begin{aligned}
\Theta_{0,0} &=w^0 \times w^0 \times w^0 \times .... \times w^0 \\
\Theta_{0,1} &=v^0 \times a_1^0 \times a_1^0 \times .... \times a_1^0 \\
\Theta_{0,m-1} &=v^0 \times \bar{a_1^0} \times \bar{a_1^0} \times .... \times \bar{a_1^0} \\
\ldots  & \ldots \\
\Theta_{0,i} &=c^0 \times c^0\times \cdots \times c^0\times v^0 \times a_1^0 \times \cdots \times a_1^0\times a_1^0 \\
\Theta_{0,m-i} &=c^0 \times c^0\times \cdots \times c^0\times \bar{v^0} \times \bar{a_1^0} \times \cdots \times \bar{a_1^0}\times \bar{a_1^0} \\
\ldots  & \ldots \\
\Theta_{0,h-1} &=c^0\times c^0\times \cdots  \cdots \times c^0 \times v^0\times a_1^0 \\
\Theta_{0,m-h+1} &=c^0\times c^0\times \cdots  \cdots \times c^0 \times\bar{v^0}\times \bar {a_1^0} \\
\Theta_{0,h} &=c^0 \times c^0 \times \cdots \cdots \times c^0 \times \tilde{v^0} \qquad (m=2h)\\
\Theta_{0,h} &=c^0 \times c^0 \times \cdots \cdots \times c^0 \times v^0 \qquad (m=2h+1)\\
\Theta_{0,h+1} &=c^0 \times c^0 \times \cdots \cdots \times c^0 \times \bar{v^0} \qquad (m=2h+1)
\end{aligned}
$$

So the N\'eron's model allows us to identify each rational component $\Theta_{0,i}$ of a singular fiber, hence to know which component is cut by a section through the surface or by the exceptional divisor of the desingularization of a double point.

\bigskip

\subsection{Some known facts concerning the $L$-series of a $K3$-surface}

For definitions concerning $K3$-surfaces, see for example \cite {Ber2}.

Let $V$ be a smooth projective variety of dimension $d$ over a finite field $k$ with $q$ elements. Suppose moreover that $V$ is geometrically irreducible (i.e. irreducible over some algebraic closure of $k$). Let $N_n$ denote the number of points of $V$ in a $k$- extension of degree $n$. Then, 
the zeta function attached to $V$ is defined by
$$Z_V(T):=\exp(\sum_{n\geq 1}N_n\frac {T^n}{n}).$$
By Weil's conjectures (1949), proved by Dwork (1960) and Deligne [see for example \cite{Hu}], we know that $Z_V(T)$ is a rational function with coefficients in $\mathbb Q$ satisfying the functional equation
$$Z_V(\frac {1}{q^d T})=\pm (q^{d/2} T)^c Z_V(T) $$
for some $c \in \mathbb N$.

If $V$ is an algebraic $K3$-surface  defined over $\mathbb Q$, then, for almost all primes $p$, its reduction modulo $p$ is again a $K3$-surface that we denote $V_p$. Moreover, one can show that $Z_{V_p}(T)$ is of the form
$$Z_{V_p}(T)=\frac {1}{(1-T)(1-p^2 T)P_2(T)}$$
where $P_2(T)$ is a polynomial of degree $22$ satisfying $P_2(0)=1$ and such that
$$P_2(T)=Q_p(T)R_p(T)$$
where $R_p(T)$ comes from the algebraic cycles with $R_p(\frac {T}{p})\in \mathbb Z[T]$ and $Q_p(T)$ comes from the transcendental 
cycles. Under some conditions, the  $H^{1,1}$ part of the Hodge decomposition of $H^2$ intersected with $H^2(V,\mathbb Z)$ has dimension $20$. In that case, the  $K3$-surface is singular, that is to say of Picard number $\rho=20$ and its transcendental lattice is of rank $2$. Hence $Q_p(T)$ is of degree $2$ for all primes
 $p$ not dividing an integer $N$. The Langlands's philosophy says that one can define the $Q_p(T)$ for $p$ dividing $N$ in such a manner that
$$Z(V,s):=\prod _p \frac {1}{Q_p(p^{-s})}=\sum_{n\geq 1}\frac {a_n}{n^s}$$
be the Dirichlet series of a cusp form.

For example \cite {SB}, if the affine equation of the surface is given by
$$t^2(x+y)(x+z)(y+z)+xyz=0,$$ 
the weight of the cusp form is $3$ and the level determined by the determinant of the transcendental lattice. In that case, the cusp form is
$$f(z)=\eta (z)^2\eta (2z)\eta (4z)\eta (8z)^2 \in S_3(\Gamma_0(8),\epsilon_8)$$
where $\epsilon_8$ is the character associated to $\mathbb Q(\sqrt{-2})$.
\bigskip

\subsection{Two results of Shioda}
\bigskip
We refer the reader to Shioda's article \cite{S1} but also to \cite{S}, \cite{SI}.

1) Let $\Phi : X \rightarrow \mathbb P^1$ be an elliptic surface with a section and consider the sections of this elliptic fibration  i.e. determined by the rational points of the corresponding elliptic curve defined over the field of rational functions in $s$.

Denote by $r(\Phi)$ the rank of the group of sections. Then the Picard number $\rho (X)$ satisfies the equation
\begin{equation}
\rho(X) = r( \Phi)+2+\sum_{\nu =1}^h (m_{\nu}-1)
\end{equation}
where $h$ is the number of singular fibers and $m_{\nu}$ the number of irreducible components of the corresponding singular fiber.

\bigskip

2) Let $(S,\Phi, \mathbb P^1)$ be an elliptic surface with a section $\Phi$, without exceptional curves of first kind.

Denote by $NS(S)$ the group of algebraic equivalence classes of divisors of $S$.

Let $u$ be the generic point of $\mathbb P^1$ and $\Phi^{-1}(u)=E$ the elliptic curve defined over $K=\mathbb C(u)$ with a $K$-rational point $o=o(u)$. Then, $E(K)$ is an abelian group of finite type provided that $j(E)$ is transcendental over $\mathbb C$.

Let $r$ be the rank of $E(K)$ and $s_1,..., s_r$ be generators of $E(K)$ modulo torsion. Besides, the torsion group  $E(K)_{tors}$ is generated by at most two elements $t_1$ of order $e_1$ and $t_2$ of order $e_2$ such that $1\leq e_2$, $e_2 |e_1$ and $\mid E(K) _{tors}\mid =e_1e_2$.

The group $E(K)$ of $K$-rational points of  $E$ is canonically identified with the group of sections of $S$ over $\mathbb P^1(\mathbb C)$.

For $s\in E(K)$, we denote by $(s)$ the curve image in $S$ of the section corresponding to $ s$.

Let us define
$$D_{\alpha}:=(s_{\alpha})-(o) \,\,\,\,\,\,1\leq\alpha \leq r$$
$$D'_{\beta}:=(t_{\beta})-(o) \,\,\,\,\,\,\beta =1,2.
$$
Consider now the singular fibers of $S$ over
$\mathbb P^1$. We set
$$\Sigma:=\{v\in \mathbb P^1  / C_v=\Phi^{-1}(v)\,\,\,\, {\hbox {be a singular fiber}} \}$$
and for each $v\in \Sigma$, $\Theta_{v,i}$, $0\leq i \leq m_v-1$, the
$m_v$ irreducible components of $C_v$.

Let $\Theta_{v,0}$ be the unique component of $C_v$ passing through $o(v)$.

One gets
$$C_v=\Theta_{v,0} +\sum_{i\geq 1}\mu_{v,i}\Theta_{v,i},\,\,\,\,\,\,\,\mu_{v,i}\geq 1.$$
Let $A_v$ be the matrix of order $m_v-1$ whose entry of index
$(i,j)$ is $(\Theta_{v,i}\Theta_{v,j})$, $i,j\geq 1$, where $(DD')$
is the intersection number of the divisors $D$ et $D'$ along
$S$. Finally $f$ will denote a non singular fiber,
i.e. $f=C_{u_0}$ for $u_0\notin \Sigma$.

\begin{theo}
The N\'eron-Severi group $NS(S)$ of the elliptic surface $S$
is generated by the following divisors
$$f, \Theta_{v,i} \,\,\,\,\,\,(1\leq i \leq
m_v-1,\,\,\,\,v\in\Sigma)$$
$$(o), D_{\alpha}\,\,\,\,\,\,1\leq \alpha \leq r,
\,\,\,\,D'_{\beta}\,\,\,\,\beta =1,2.$$

The only relations between these divisors are at most two relations 
$$e_{\beta}D'_{\beta}\approx e_{\beta}(D'_{\beta} (o))f+\sum_{v\in
\Sigma
}(\Theta_{v,1},...,\Theta_{v,m_v-1})e_{\beta}A_v^{-1}\left ( \begin{array}{l}
(D'_{\beta}\Theta_{v,1})\\.\\
.\\
.\\
(D'_{\beta}\Theta_{v,m_v-1}
)
\end{array}
\right )
$$

where $\approx$ stands for the algebraic equivalence.

\end{theo}
\bigskip

\section{Proof of theorem 1}
\subsection{Determinant of the transcendental lattice}
The polynomial $P_{10}$ defines a non-smooth surface $S_{10}$ which gives after desingularization the K3-surface $Y_{10}$.

\begin{theo}
The determinant of the transcendental lattice of $Y_{10}$ is equal to $72$.
\end{theo}

\begin{proof}

To compute the determinant of the transcendental lattice of $Y_{10}$  we choose an elliptic fibration to build a N\'eron model and apply Shioda's results.
The proof however involves several lemmas useful later on.

The $K3$-surface $Y_{10}$ is a double cover of the Beauville's rational elliptic surface defined by
$$(x_1+y_1)(x_1+z_1)(y_1+z_1)+ux_1y_1z_1=0$$

Let us recall its singular fibers:

$$\begin{array}{clcc}
{\hbox {at}} & u=\infty  & {\hbox {of type}} & I_6\\
{\hbox {at}} & u=0  & {\hbox {of type}} & I_3\\
{\hbox {at}} &  u=1  & {\hbox {of type}} &  I_2\\
{\hbox {at}} &  u=-8   & {\hbox {of type}} &  I_1.
\end{array}$$

Now cutting $Y_{10}$ by the hyperplane $t=s(x+y+z)$, you get, after simplification, the corresponding elliptic fibration
$$s^2(x+y)(x+z)(y+z)+(s^2-10s+1)xyz=0.$$

Thus, putting $u=(s^2-10s+1)/s^2$, we deduce from above, the singular fiber structure of this elliptic fibration of the $K3$-hypersurface $Y_{10}$:

$$\begin{array}{clcc}
{\hbox {at}} &  s=0  & {\hbox {of type}} & I_{12}\\
{\hbox {at}} &  s=\infty  & {\hbox {of type}} & I_2\\

{\hbox {at}} & s=1/10  & {\hbox {of type}} & I_2\\
{\hbox {at}} &  s=\alpha\qquad (\alpha^2-10\alpha +1=0) & {\hbox {of type}} &  I_3\\
{\hbox {at}} &  s=\beta\qquad (\beta^2-10\beta+1=0) & {\hbox {of type}} &  I_3\\
{\hbox {at}} &  s=1 \qquad  & {\hbox {of type}} &  I_1\\
{\hbox {at}} &  s=1/9 \qquad    & {\hbox {of type}} &  I_1.
\end{array}
$$
From (4) since $\rho=20$ (\cite{PS}), it follows $r=1$ and in order to describe the N\'eron -Severi group, we need an infinite section on the surface.

The singular surface $S_{10}$ contains the $7$ double points

$$P_{01}=(1:0:0:0) \quad P_{02}=(0:1:0:0) \quad P_{03}=(0:0:1:0) \quad P_{04}=(0:0:0:1)$$
$$P_{12}=(1:-1:0:0) \qquad P_{13}=(1:0:-1:0) \qquad P_{23}=(0:1:-1:0).$$
The double points $P_{12}$, $P_{13}$, $P_{23}$ lie in all the singular fibers.
The double points $P_{01}$, $P_{02}$, $P_{03}$ lie only in the singular fiber above $0$.

The  $4$ lines $P_{01}P_{03}P_{13}$, $P_{02}P_{03}P_{23}$, $P_{01}P_{02}P_{12}$, $P_{12}P_{13}P_{23}$ of respective equations
$$y=0 \quad t=0 \qquad x=0 \quad t=0 \qquad z=0 \quad t=0 \qquad x+y+z=0 \quad t=0$$
passing each through three double points and the lines $P_{03}P_{04}$, $P_{01}P_{04}$, $P_{02}P_{04}$ of respective equations
$$x=0 \quad y=0 \qquad y=0 \quad z=0 \qquad x=0 \quad z=0$$
passing each through two double points, lie on the surface.

We shall complete the generators of $NS(Y_{10})$ with the infinite section $\Sigma$.
\subsection{The Weierstrass model $E_s$ and the infinite section $\Sigma$}
Now we need a suitable Weierstrass model to apply N\'eron's desingularization.

Starting with the cubic projective equation $C_s$
$$s^2(x+y)(x+z)(y+z)+(s^2-10s+1)xyz=0,$$
we use a projective transformation sending the flex point $(1:-1:0)$ with flex tangent
$$s^2(x+y)+(s^2-10s+1)z=0$$
to the flex point $(0:1:0)$ with flex tangent $Z'=0$. 

This birational transformation
$$\begin{array}{lll}
X' & = & s^4(10s-1)(x+y)\\
Y' & = & -s^4(10s-1)(s^2-10s+1)y\\
Z' & = &s^2(x+y)+(s^2-10s+1)z
\end{array}
$$
gives the Weierstrass model

$$Y'^2Z'+(s^2-10s+1)X'Y'Z'=X'^3+X'^2Z's^2(1-10s-s^2)+X'Z'^2(10s^7-s^6).$$

In this model, we find with Pari the torsion points 
$$\begin{array}{c}
s_6=(s^2(10s-1):0:1)\\
5s_6=(-s^2+10s^3:-s^2(10s-1)(s^2-10s+1):1)\\
2s_6=(s^4:0:1)\\
4s_6=(s^4:-s^4(s^2-10s+1):)\\
3s_6=(0:0:1)\\
(0).
\end{array}
$$

The infinite section $\Sigma$ was found by Lecacheux \cite{Le} and is defined by the point

$$X'=-\frac {1}{432} s(2s-1)^2(8s-1)^2(5s-1)^2$$
$$ Y'=-\frac{1}{106251264}(1-\sqrt{-3}+s(19\sqrt{-3}-3)-s^2(78+82\sqrt{-3})+160s^3)$$
$$(13+5\sqrt{-3}-s(145+37\sqrt{-3})+244s^2)(8s-1)(28s-5-\sqrt{-3})(41\sqrt{-3}-9)(2s-1)(5s-1)$$
$$Z'=s^3.$$

Thus $\Sigma$ is defined over $\mathbb Q(\sqrt{-3})$.

\subsection{The N\'eron's model over $s=0$}
Using a transformation as in theorem 3, we obtain the N\'eron-Weierstrass model $\mathcal E_s$ over $s=0$
$$\begin{array}{l}
Y^2Z+XYZ(s^2-12s+1)+YZ^2(2s^8-24s^7)  =\\
X^3+X^2Z(s-10s^2-9s^3-s^4+6s^6)+XZ^2(2s^7-39s^8-36s^9-4s^{10}+12s^{12})\\
+Z^3(-s^{12}-38s^{14}-36s^{15}-4s^{16}+8s^{18})
\end{array}
$$

with
$$
\begin{array}{l}
X=(-s^4+10s^5-2s^8)(x+y)-2s^6(s^2-10s+1)z \\
Y=(-s^5+10s^6+s^8)x+(s^4-21s^5+111s^6-10s^7+s^8)y+(s^6-10s^7+s^8)z\\
Z=s^2(x+y)+(s^2-10s+1)z.
\end{array}
$$

In this model we have $\nu_0=0$, $\bar{\nu_0}=-1$ and the equation of the conic is
$$Y^2+XY+Z^2=0.$$
Thus all the rational components of the N\'eron's model over $0$ are defined over $\mathbb Q$ and the point $(-2:1:1)$ is on the conic.
\subsubsection{The infinite section}
It corresponds to the point on $\mathcal E_s$
$$\begin{array}{l}
X=-\frac {1}{432} (1-30s+357s^2-2140s^3+6756s^4-10560s^5+6400s^6+864s^8)s\\
Y=-\frac {1}{15552}((130752-524800\sqrt{-3})s^9+(1013760\sqrt{-3}-1572480)s^8\\
+(-611640\sqrt{-3}+2517768)s^7+(58776\sqrt{-3}-1687896)s^6\\
+(590274+80550\sqrt{-3})s^5+(-116172-37872\sqrt{-3})s^4+(7665\sqrt{-3}+12924)s^3\\
+(-819\sqrt{-3}-756)s^2+(45\sqrt{-3}+18)s-\sqrt{-3})\\
Z=s^3
\end{array}
$$

Thus $\Sigma$ cuts $\Theta_{0,0}$ at the intersection point of $\Theta_{0,0}$ with the zero section $(0)$.

Now we proceed to describe which component cut the torsion points.

\subsubsection{Torsion sections}
The correspondance between the $E_s$-model $(X':Y':Z')$ and the $\mathcal E_s$ model $(X:Y:Z)$ is given by
$$X=X'-2s^6Z' \qquad \qquad Y=Y'+sX'+s^6Z' \qquad \qquad Z=Z'.$$

We find that the $6$-torsion point $s_6$ corresponds in the $\mathcal E_s$ model to the point
$$(-s^2+10s^3-2s^6:-s^3+10s^4+s^6:1).$$
So it is identified in the N\'eron model with the point
$$(0:0:1) \times (-1:0:0) \times (-1:0:0)\times (-1:0:0) \times (-1:0:0) \times (-1:0:0),$$
hence belongs to $\Theta_{0,2}$.

So, we obtain 
$$5s_6 \in \Theta_{0,10}$$
$$2s_6 \in \Theta_{0,4}$$
$$3s_6 \in \Theta_{0,6}$$
$$4s_6 \in \Theta_{0,8}$$
$$(0) \in \Theta_{0,0}.$$

\subsubsection{The exceptional divisors}

\begin{itemize}
\item The exceptional divisor $(P_{01})$ cuts $\Theta_{0,2}$ at the same point as the torsion section  $(s_6)$.
\item The exceptional divisor $(P_{13})$ cuts $\Theta_{0,4}$ at the same point as the torsion section $(2s_6)$.
\item The exceptional divisor $(P_{03})$ cuts $\Theta_{0,6}$ at the same point as the torsion section $(3s_6)$.
\item The exceptional divisor $(P_{23})$ cuts $\Theta_{0,8}$ at the same point as the torsion section $(4s_6)$.
\item The exceptional divisor $(P_{02})$ cuts $\Theta_{0,10}$ at the same point as the torsion section $(5s_6)$.
\item The exceptional divisor $(P_{12})$ cuts $\Theta_{0,0}$ at the same point as the torsion section $(0)$.
\end{itemize}
Thus the exceptional divisors are exactly the torsion sections.

\subsection{The N\'eron's model over $s=\infty$}
In that case, we put $s=1/\sigma $; so we get the projective cubic $C_{\sigma}$

$$(x+y)(x+z)(y+z)+(\sigma^2-10\sigma+1)xyz=0.$$
Now, using the change variables
$$
\begin{array}{lll}
  x & = &-(\sigma^2-10\sigma+1)X'+Y' \\
y & = & -Y'\\
z & = &X'+\sigma(\sigma -10)Z',
\end{array}
$$
$$
\begin{array}{lll}
  X' & = &-(x+y)(\sigma^2-10\sigma) \\
Y' & = & -y(\sigma^2-10\sigma)(\sigma^2-10\sigma+1)\\
Z' & = &x+y+(\sigma^2-10\sigma+1)z,
\end{array}
$$

we obtain the $E_{\sigma}$-model. 
$$Y'^2Z'-(\sigma^2-10\sigma+1)X'Y'Z'=X'^3+X'^2Z'(\sigma^2-10\sigma-1)+(10\sigma-\sigma^2)X'Z'^2.$$ 

And, by the change variables

\begin{align}\label{E:mm3}
X' & =  \frac {X}{9}+\frac {20}{3} \sigma Z                   & X & =  9X'-60\sigma Z' \notag\\
Y' & = - \frac {X}{9} -\frac {Y}{27} +\frac {10}{3} \sigma Z   & Y & =  -27X'-27Y'+270\sigma Z' \notag \\
Z' & =  Z                                                     & Z & =  Z'\notag
\end{align}
we get the N\'eron's model ${\mathcal E}_{\sigma}$
\begin{align}
Y^2Z+XYZ(9-30\sigma +3\sigma^2)+ & YZ^2180 \sigma^2(\sigma -10)\notag \\
 & = X^3-(27-180\sigma)X^2Z+XZ^2\sigma^2(810\sigma +2619)\notag \\
 &+Z^3(24300\sigma^2-274860\sigma^3+48600\sigma^4).\notag
\end{align}
We have the relations
\begin{align}
X & = -60\sigma(\sigma^2-10\sigma+1)z+(-9(\sigma^2-10\sigma)-60\sigma)(x+y)\notag\\
Y &=27(x+y)\sigma^2+27y(\sigma^2-10\sigma)(\sigma^2-10\sigma+1)+270z\sigma(\sigma^2-10\sigma+1)\notag\\
Z &=x+y+z(\sigma^2-10\sigma+1).\notag
\end{align}

In this model we get $\nu_0=(-9+3\sqrt{-3})/2$, $ \bar{\nu_0}=(-9-3\sqrt{-3})/2 $ and the equation of the conic is
$$Y^2+9XY+27X^2-24300Z^2=0.$$

\subsubsection{The infinite section $\Sigma$}

The section $\Sigma$ cuts $E_{\sigma}$ at
\begin{align}
X' & = -\frac {1}{432} (\sigma -5)^2(\sigma -2)^2(\sigma -8)^2\notag\\
Y' &=\frac {1}{15552}\sqrt{-3}(\sigma^3+\sigma^2(-15+4\sqrt{-3})+\sigma(42-40\sqrt{-3})+40+4\sqrt{-3})\notag\\
 &(\sigma^2+\sigma(-10+\sqrt{-3})+13-5\sqrt{-3})(\sigma-5+\sqrt{-3})(\sigma-5)(\sigma-2)(\sigma-8)\notag\\
Z' &=1\notag
\end{align}
which gives in the N\'eron model the point
$$(-\frac {400}{27},-\frac {200}{27}+\frac {8200}{243}\sqrt{-3},0)\times(-\frac {400}{27},-\frac {200}{27}+\frac {8200}{243}\sqrt{-3},0).$$
Thus the infinite section $\Sigma$ cuts $\Theta_{\infty,0}$.

\subsubsection{Torsion sections}
In the $E_{\sigma}$-model $(X':Y':Z')$, there are two $6$-torsion sections $s_6=(10\sigma-\sigma^2:(10\sigma-\sigma^2)(\sigma^2-10\sigma+1):1)$, $5s_6=(10\sigma -\sigma^2:0:1)$, two $3$-torsion sections $4s_6=(1:0:1)$, $2s_6=(1:\sigma^2-10\sigma+
1:1)$, one $2$-torsion section $3s_6=(0:0:1)$ and the zero section $(0)=(0:1:0)$.
We find that the $6$-torsion point $s_6$ corresponds in the $\mathcal E_{\sigma}$ model to the point
$$(30\sigma-9\sigma^2:27\sigma(-10+102\sigma-20\sigma^2+\sigma^3):1).$$
So it is identified in the N\'eron model with the point
$$(30:-270:1) \times (30:-270:0)$$
hence belongs to $\Theta_{\infty,1}$, since the point $(30:-270:1)$ lies on the conic.

Similarly, we obtain 
$$5s_6=(30\sigma-9\sigma^2:27\sigma^2:1) \rightsquigarrow (30:0:1)\times(30:0:0)\in \Theta_{\infty,1}$$
$$4s_6 =(9-60\sigma:-54+540\sigma-27\sigma^2:1)\rightsquigarrow (9:-54:0)\times (9:-54:0) \in \Theta_{\infty,0}$$
$$2s_6=(9-60\sigma:-27+270\sigma:1)\rightsquigarrow (9:-27:0) \times (9:-27:0) \in \Theta_{\infty,0}$$
$$3s_6=(-60\sigma:270\sigma:1)\rightsquigarrow (-60:270:1)\times (-60:270:0) \in \Theta_{\infty,1}$$
$$(0)=(0:1:0)\rightsquigarrow (0:1:0) \times (0:1:0) \in \Theta_{\infty,0}$$

\subsubsection{The exceptional divisors}

\begin{itemize}
\item The exceptional divisor $(P_{13})$ cuts $\Theta_{\infty,0}$ at the point $(3(20\sigma-3),-27(10\sigma-1),-1)\rightsquigarrow (-9,27,0)\times(-9,27,0)$.
\item The exceptional divisor $(P_{23})$ cuts $\Theta_{\infty,0}$ at the point $(3(20\sigma-3),27(\sigma^2-20\sigma+2,-1)\rightsquigarrow (-9,54,0)\times(-9,54,0)$.
\item The exceptional divisor $(P_{12})$ cuts $\Theta_{\infty,0}$ at the same point as the torsion section $(0)$.
\item The exceptional divisors $(P_{01})$ cuts $\Theta_{\infty,1}$ at the point $(-3\sigma(3\sigma-10),27\sigma^2,1)\rightsquigarrow (30,0,1)\times(30,0,0)$, the same point as $5s_6$.
\item The exceptional divisors $(P_{02})$ cuts $\Theta_{\infty,1}$ at the point $(-3\sigma(3\sigma-10),27\sigma(-10+102\sigma-20\sigma^2+\sigma^3),1)\rightsquigarrow (30,-270,1)\times(30,-270,0)$, the same point at $s_6$.
\item The exceptional divisors $(P_{03})$ cuts $\Theta_{\infty,1}$ at the point $(-60\sigma,270\sigma,1)\rightsquigarrow (-60,270,1)\times(-60,270,0)$, the same point as $3s_6$.
\end{itemize}

\subsection{The singular fiber $I_3$ over $s=\alpha$}
The rational components are the following
\begin{itemize}
\item $\Theta_{{\alpha},0}$ is the line defined by $(x+y=0, \quad t=\alpha z)$.
\item $\Theta_{{\alpha},1}$ is the line defined by $(y+z=0, \quad t=\alpha x)$.
\item $\Theta_{{\alpha},2}$ is the line defined by $(x+z=0, \quad t=\alpha y)$.
\end{itemize}
These components are defined over $\mathbb Q(\sqrt{6})$.
\begin{itemize}
\item The section $(s_6)$ cuts $\Theta_{\alpha,1}$.
\item The section $(2s_6)$ and the exceptional divisor$(P_{13})$ cut $\Theta_{\alpha,2}$ at the same point.
\item The section $(3s_6)$ cuts $\Theta_{\alpha,0}.$
\item The section $(4s_6)$ and the exceptional divisor $(P_{23})$ cut $\Theta_{\alpha,1}$ at the same point.
\item The section $(5s_6)$ cuts $\Theta_{\alpha,2}$.
\item The section $(0)$ and the exceptional divisor $(P_{12})$ cut $\Theta_{\alpha,0}$ at the same point.
\end{itemize}
The infinite section $\Sigma$ gives a point $(x:y:z)$ on $C_s$ satisfying
$$x+y=-\frac {1}{432} \frac {(s^2-10s+1)(2s-1)^2(5s-1)^2(8s-1)^2}{s^2},$$
thus cuts $\Theta_{\alpha,0}$.
\subsection{The singular fiber $I_3$ over $s=\beta$}
The rational components are the following
\begin{itemize}
\item $\Theta_{{\beta},0}$ is the line defined by $(x+y=0, \quad t=\beta z)$.
\item $\Theta_{{\beta},1}$ is the line defined by $(y+z=0, \quad t=\beta x)$.
\item $\Theta_{{\beta},2}$ is the line defined by $(x+z=0, \quad t=\beta y)$.
\end{itemize}
These components are defined over $\mathbb Q(\sqrt{6})$.
\begin{itemize}
\item The section $(s_6)$ cuts $\Theta_{\beta,1}$.
\item The section $(2s_6)$ and the exceptional divisor$(P_{13})$ cut $\Theta_{\beta,2}$ at the same point.
\item The section $(3s_6)$ cuts $\Theta_{\beta,0}$.
\item The section $(4s_6)$ and the exceptional divisor $(P_{23})$ cut $\Theta_{\beta,1}$ at the same point.
\item The section $(5s_6)$ cuts $\Theta_{\beta,2}$.
\item The section $(0)$ and the exceptional divisor $(P_{12})$ cut $\Theta_{\beta,0}$ at the same point.
\end{itemize}
As previously, the infinite section $\Sigma$ cuts $\Theta_{\beta,0}$.
\subsection{The singular fiber $I_2$ over $s=1/10$}
The component $\Theta_{1/10,0}$ is defined by $x+y+z=0 \qquad t=0$.

The component $\Theta_{1/10,1}$ is defined by $x+y+z=10t \qquad xy+xz+yz=0$.

These two components are defined over $\mathbb Q.$

The section $(2s_6)$ and the exceptional divisor$(P_{13})$ cut $\Theta_{1/10,0}$ at the same point.

The section $(4s_6)$ and the exceptional divisor $(P_{23})$ cut $\Theta_{1/10,0}$ at the same point.

The section $(0)$ and the exceptional divisor $(P_{12})$ cut $\Theta_{1/10,0}$ at the same point.

The sections $(s_6)$, $(3s_6)$ and $(5s_6)$ cut $\Theta_{1/10,1}$.

The infinite section $\Sigma$ gives a point $(x:y:z)$ on $C_s$ satisfying
$$x+y+z=\frac {1}{2951424s^2}(13-5\sqrt{-3}+(-145+37\sqrt{-3})s+244s^2)$$
$$(13+5\sqrt{-3}+(-145-37\sqrt{-3})s+244s^2)(28s-5+\sqrt{-3})(28s-5-\sqrt{-3})(10s-1),$$
thus cuts $\Theta_{1/10,0}$.

\bigskip

\subsection{The Gram matrix of the N\'eron-Severi lattice}

We shall see that the N\'eron-Severi lattice is generated by 

$$(o), \,\,\,f,\,\,\,(\Sigma),\,\,\,(s_6),\,\,\,\,\Theta_{0,i}, 1\leq i \leq 11, \,\,\,\Theta_{1/10,1},\,\,\,\Theta_{\infty,1},\,\,\, \Theta_{\alpha,1},\,\,\Theta_{\alpha,2},\,\,\Theta_{\beta,1},\,\,\Theta_{\beta,2}.$$

Moreover the only relation between these generators is computed to be

$$6((s_6)-(0))=12f-(5\Theta_{0,1}+10\Theta_{0,2}+9\Theta_{0,3}+8\Theta_{0,4}+7\Theta_{0,5}+6\Theta_{0,6}+5\Theta_{0,7}$$
$$4\Theta_{0,8}+3\Theta_{0,9}+6\Theta_{0,10}+\Theta_{0,11})-3\Theta_{\infty,1}-4\Theta_{\alpha,1}-2\Theta_{\alpha,2}-4\Theta_{\beta,1}-2\Theta_{\beta,2}-3\Theta_{1/10,1}.$$
To get the intersection numbers of the divisors,
we use the following result \cite{Sha}.

If $D$ denotes a divisor on the surface $V$, with arithmetic genus $p_a(D)$, then
$$p_a(D)=\frac {(D.(D+K))}{2}+1.$$
Moreover if the surface $V$ is $K3$, the class $K$ of the canonical bundle is $0$.

First we express the intersection matrix  of the divisors in the following order 
$$(o), \,\,\,f,\,\,\,(\Sigma),\,\,\,\,\Theta_{0,i}, 1\leq i \leq 11, \,\,\,\Theta_{1/10,1},\,\,\,\Theta_{\infty,1},\,\,\, \Theta_{\alpha,1},\,\,\Theta_{\alpha,2},\,\,\Theta_{\beta,1},\,\,\Theta_{\beta,2}.$$

$$
\left (
\begin{smallmatrix}
-2 & 1 & 1 & 0 & 0 & 0 & 0 & 0 & 0 & 0 & 0 & 0 & 0 & 0 & 0 & 0 & 0 & 0 & 0 & 0\\
1 & 0 & 1 & 0 & 0 & 0 & 0 & 0 & 0 & 0 & 0 & 0 & 0 & 0 & 0 & 0 & 0 & 0 & 0 & 0 \\
0 & 1  & -2 & 0 & 0 & 0 & 0 & 0 & 0 & 0 & 0 & 0 & 0 & 0 & 0 & 0 & 0 & 0 & 0 & 0 \\
0 & 0 & 0 & -2 & 1 & 0 & 0 & 0 & 0 & 0 & 0 & 0 & 0 & 0 & 0 & 0 & 0 & 0 & 0 & 0 \\
0 & 0 & 0 & 1 & -2 & 1 & 0 & 0 & 0 & 0 & 0 & 0 & 0 & 0 & 0 & 0 & 0 & 0 & 0 & 0 \\
0 & 0 & 0 & 0 & 1 & -2 & 1 & 0 & 0 & 0 & 0 & 0 & 0 & 0 & 0 & 0 & 0 & 0 & 0 & 0 \\
0 & 0 & 0 & 0 & 0 & 1 & -2 & 1 & 0 & 0 & 0 & 0 & 0 & 0 & 0 & 0 & 0 & 0 & 0 & 0 \\
0 & 0 & 0 & 0 & 0 & 0 & 1 & -2 & 1 & 0 & 0 & 0 & 0 & 0 & 0 & 0 & 0 & 0 & 0 & 0 \\
0 & 0 & 0 & 0 & 0 & 0 & 0 & 1 & -2 & 1 & 0 & 0 & 0 & 0 & 0 & 0 & 0 & 0 & 0 & 0 \\
0 & 0 & 0 & 0 & 0 & 0 & 0 & 0 & 1 & -2 & 1 & 0 & 0 & 0 & 0 & 0 & 0 & 0 & 0 & 0 \\
0 & 0 & 0 & 0 & 0 & 0 & 0 & 0 & 0 & 1 & -2 & 1 & 0 & 0 & 0 & 0 & 0 & 0 & 0 & 0 \\
0 & 0 & 0 & 0 & 0 & 0 & 0 & 0 & 0 & 0 & 1 & -2 & 1 & 0 & 0 & 0 & 0 & 0 & 0 & 0 \\
0 & 0 & 0 & 0 & 0 & 0 & 0 & 0 & 0 & 0 & 0 & 1 & -2 & 1 & 0 & 0 & 0 & 0 & 0 & 0 \\
0 & 0 & 0 & 0 & 0 & 0 & 0 & 0 & 0 & 0 & 0 & 0 & 1 & -2 & 0 & 0 & 0 & 0 & 0 & 0 \\
0 & 0 & 0 & 0 & 0 & 0 & 0 & 0 & 0 & 0 & 0 & 0 & 0 & 0 & -2 & 0 & 0 & 0 & 0 & 0 \\
0 & 0 & 0 & 0 & 0 & 0 & 0 & 0 & 0 & 0 & 0 & 0 & 0 & 0 & 0 & -2 & 0 & 0 & 0 & 0 \\
0 & 0 & 0 & 0 & 0 & 0 & 0 & 0 & 0 & 0 & 0 & 0 & 0 & 0 & 0 & 0 & -2 & 1 & 0 & 0 \\
0 & 0 & 0 & 0 & 0 & 0 & 0 & 0 & 0 & 0 & 0 & 0 & 0 & 0 & 0 & 0 & 1 & -2 & 0 & 0 \\
0 & 0 & 0 & 0 & 0 & 0 & 0 & 0 & 0 & 0 & 0 & 0 & 0 & 0 & 0 & 0 & 0 & 0 & -2 & 1 \\
0 & 0 & 0 & 0 & 0 & 0 & 0 & 0 & 0 & 0 & 0 & 0 & 0 & 0 & 0 & 0 & 0 & 0 & 1 & -2 
\end{smallmatrix}
\right) 
$$

This matrix has determinant $-6^2\times 72$.

Then we prove the following lemma.

\begin{lem}
The N\'eron-Severi group $NS( {Y_{10}})$ is generated by the classes of the following divisors:
$$(o), \,\,\,f,\,\,\,(\Sigma),\,\,(s_6),\,\,\Theta_{0,i}, 1\leq i \leq 11, \,\,\,\Theta_{1/10,1},\,\,\,\Theta_{\infty,1},\,\,\, \Theta_{\alpha,1},\,\,\Theta_{\alpha,2},\,\,\Theta_{\beta,1}.\,\,\Theta_{\beta,2}.$$
The only relation between these classes is:
$$6((s_6)-(0))=12f-(5\Theta_{0,1}+10\Theta_{0,2}+9\Theta_{0,3}+8\Theta_{0,4}+7\Theta_{0,5}+6\Theta_{0,6}+5\Theta_{0,7}$$
$$4\Theta_{0,8}+3\Theta_{0,9}+6\Theta_{0,10}+\Theta_{0,11})-3\Theta_{\infty,1}-4\Theta_{\alpha,1}-2\Theta_{\alpha,2}-4\Theta_{\beta,1}-2\Theta_{\beta,2}-3\Theta_{1/10,1}.$$
The determinant of the Picard lattice is equal to $-72$.

\end{lem}

\begin{proof}

Let  $\Lambda$ be the lattice generated by the classes of the divisors given in the lemma and $\Lambda'$ the sublattice generated by the same divisors minus $(s_6)$. From the computation above
$$\det \Lambda'=-6^2\times 72.$$
But the only relation between these classes is the one given previously. So $\Lambda'$ has index $6$ in $\Lambda$. Now either $\Lambda$ is the full N\'eron-Severi lattice or is of index $3$ in it. In the latter case, there would be an element of $(\Sigma)+k(s_6), \,\,\,k=0,...,5$ which is $3$-divisible in $NS({ Y_{10}})$. But taking the restrictions to the fibers over $0$ and $\infty$ rules out that possibility for $k\geq 1$.

If $\Sigma =3Q$, then either 
$$Q\,\, {\hbox{cuts}}\,\,\Theta_{0,0},\,\Theta_{\infty,0},\,\Theta_{\alpha,1},\,\Theta_{\beta,1},\,\Theta_{1/10,0},$$
or
$$Q\,\, {\hbox{cuts}}\,\,\Theta_{0,4},\,\Theta_{\infty,0},\,\Theta_{\alpha,1},\,\Theta_{\beta,1},\,\Theta_{1/10,0},$$
or cuts equivalent curves for the Gram matrix. But in all these cases the determinant of the Gram matrix is different from $-8$.
Hence $\det NS({Y_{10}})=-72$.

\end{proof}
\begin{rem}
Following Peters and Stienstra, \cite{PS} remark 7, the generic member of the family with $\rho =19$ has a transcendental lattice $T$ with basis $\{e_0,e_1,e_2 \}$ and Gram matrix
$\left(
\begin{matrix}
0 & 0 & 1\\
0 & 12 & 0\\
1 & 0 & 0 
\end{matrix}
\right)
$. For special members of the family with $\rho =20$, there exists a vector $pe_0+qe_1+re_2$, $p,\,q,\,r\, \in \mathbb Z$ which becomes algebraic. If $\tau$ is the corresponding modulus of the special member, this property is equivalent to the relation $-6p\tau^2+12q\tau+r=0$.

In our case the $\tau$ corresponding to $Y_{10}$ satisfies the relation $2\tau^2+1=0$ and $e'_0=e_0-3e_2$ becomes algebraic. Now take $T'$ a sublattice of $T$ with orthogonal basis $\{e'_0,e'_1,e'_2\}$ such that
$$
\begin{aligned}
e'_0= & e_0-3e_2\\
e'_1 = & e_1\\
e'_2 = & e_0+3e_2.
\end{aligned}
$$
Computing its Gram matrix, we find
$$Gram(e'_0,e'_1,e'_2)=\left(
\begin{matrix}
-6 & 0 & 0\\
0 & 12 & 0\\
0 & 0 & 6
\end{matrix}
\right).
$$
Finally, $\{e'_1,e'_2\}$ is a sublattice of $T_{Y_{10}}$ with Gram matrix
$$Gram(e'_1,e'_2)=\left(
\begin{matrix}
12  & 0\\
0 & 6
\end{matrix}
\right).
$$
Thus $\det(T_{Y_{10}})=72 {\hbox{ or }} 8$.

\end{rem}

\begin{lem}
\begin{enumerate}
\item Over the complex field, all the singular fibers are of type $I_n$.
\item Under the reduction modulo $p$, $p \neq 2$, $p \neq 3$, the singular fibers are of the same type. Their components are defined over $\mathbb F_p$ if $(\frac {6}{p})=1$. If  $(\frac {6}{p})=-1$, the Frobenius exchanges the corresponding fibers of type $I_3$ component by component. 
\end{enumerate}
\end{lem}

\bigskip
 \subsection{The reductions of $Y_{10}$ modulo $2$ and $3$}
The reduction of $Y_{10}$ modulo $2$ and $3$ is singular. By abuse of notation, when talking of the reduction modulo $2$ or $3$, we mean its desingularization.
\begin{theo}

\begin{enumerate}
\item The reduction of $Y_{10}$ modulo $2$ has singular fibers of type $I_{12}$, $I_4$, $IV^*$ over respectively $0$, $\infty$, $1$. Over the field $\bar{\mathbb F_2}$, it is a supersingular extremal elliptic $K3$-surface of discriminant $-2^2$, hence of Artin invariant $1$.
\item The reduction of $Y_{10}$ modulo $3$ has singular fibers of type $I_{12}$ over $0$, $I_6$ over $-1$ and $III$ over $\infty$ and $1$. 
  
\end{enumerate}
\end{theo}

\begin{proof}

1) In characteristic $2$, the discriminant is equal to $s^{12}(s-1)^8$. So there are singular fibers over $0$, $\infty$ and $1$.
\bigskip
Above $0$, the Weierstrass model is
$$Y'^2Z'+X'Y'Z'(s^2+1)=X'^3+X'^2Z's^2(1-s^2)-X'Z'^2s^6$$
The change variables $X'=X+s^7Z$, $Y'=Y+s^6Z$, $Z'=Z$  gives the model
$$
\begin{array}{l}
Y^2Z+(s^2+1)XYZ+(s^7+s^9)YZ^2= X^3+(s^2+s^4+s^7)X^2Z\\
+(s^8+s^{14})XZ^2 +(s^{12}+s^{15}+s^{16}+s^{18}+s^{21})Z^3
\end{array}
$$
satisfying $v(\gamma)=12$, $v(\beta)=8>6$, $v(\mu)=7>6$ and $v(\bar{\alpha})=v(\lambda^2+4\alpha)=0$.
Thus, by N\'eron's classification, it is an $I_{12}$ fiber, defined over $\mathbb F_2$ since $Y^2+\lambda_0 XY-\alpha_0 X^2=Y(Y+X)$.

\bigskip

Above $\infty$, setting $s=1/\sigma$, the Weierstrass model is
$$Y'^2Z'+X'Y'Z'(\sigma^2+1)=X'^3+X'^2Z'(1+\sigma^2)-X'Z'^2\sigma^2$$
The change variables $X'=X+s^3Z$, $Y'=Y+s^2Z$, $Z'=Z$ gives the model
$$
\begin{array}{l}
Y^2Z+(\sigma^2+1)XYZ+(\sigma^3+\sigma^5)YZ^2=X^3+(1+\sigma^2+\sigma^3)X^2Z\\
+(\sigma^4+\sigma^6)XZ^2+(\sigma^4+\sigma^6+\sigma^7+\sigma^8+\sigma^9)Z^3
\end{array}
$$
satisfying $v(\gamma)=4$, $v(\beta)=4>2$, $v(\mu)=3>2$ and $v(\bar{\alpha})=v(\lambda^2+4\alpha)=0$.
Thus, by N\'eron's classification, it is an $I_{4}$ fiber, defined over $\mathbb F_4$ since $Y^2+\lambda_0 XY-\alpha_0 X^2=Y^2+XY+X^2$.

\bigskip

Above $1$, we set $s+1=S$ and the Weierstrass model is
$$Y'^2Z'+X'Y'Z'S^2=X'^3+X'^2Z'S^2(1+S^2)+X'Z'^2(S+1)^6$$
The change variables $X'=X+Z$, $Y'=Y+X+S^2Z$, $Z'=Z$ gives the model
$$Y^2Z+S^2XYZ+S^2YZ^2=X^3+S^4X^2Z+S^6XZ^2+S^6Z^3$$
satisfying $v(\gamma)=6\geq 4$, $v(\beta)=6\geq 3$, $v(\alpha) \geq 2$, $v(\mu)\geq 2$ and $v(\bar{\gamma})=v(\mu^2+4\gamma)=4$.
Thus, by N\'eron's classification it is an $IV^*$ fiber, defined over $\mathbb F_4$ since the reduced cycle $A_3^0$ is defined by $Z(Y^2+\mu_2^0 YZ-\gamma_4^0 Z^2)=0$. Moreover, $\rho=2+11+3+6=22$, so the $K3$-surface is supersingular and extremal since it has a finite Mordell-Weil group. It fits Ito's table \cite{I}.

2) In characteristic $3$, the discriminant is equal to $s^{12}(s-1)^3(s-2)^6$. So there are singular fibers over $0$, $\infty$, $1$ and $2$.

Above $0$, the N\'eron's model is the reduction modulo $3$ of the N\'eron's model over $\mathbb C$
$$
\begin{array}{l}
Y^2Z+(s^2+1)XYZ-s^8YZ^2=X^3+(s-s^2-s^4)X^2Z\\
+(2s^7-s^{10})XZ^2+(-s^{12}+s^{14}-s^{16}-s^{18})Z^3
\end{array}
$$

satisfying $v(\gamma)=12$, $v(\beta)=7>6$, $v(\mu)=8>6$ and $v(\bar{\alpha})=v(\lambda^2+4\alpha)=0$.
Thus, by N\'eron's classification, it is an $I_{12}$ fiber, defined over $\mathbb F_3$ since $Y^2+\lambda_0 XY-\alpha_0 X^2=Y(Y+X)$.
\bigskip

Above $\infty$,  setting $s=1/\sigma$, the Weierstrass model is
$$Y'^2Z'+X'Y'Z'(\sigma^2-\sigma+1)=X'^3+X'^2Z'(\sigma^2-\sigma-1)+XZ'^2(\sigma-\sigma^2)$$
The change variables $X'=X+\sigma Z$, $Y'=Y+(\sigma+1) X$, $Z'=Z$ gives the model
$$
\begin{array}{l}
Y^2Z+(\sigma^2+\sigma )XYZ+(\sigma-\sigma^2+\sigma^3)YZ^2=\\
X^3-\sigma^3 X^2Z+(\sigma-\sigma^3-\sigma^4)XZ^2+(-\sigma^3+\sigma^4)Z^3
\end{array}
$$
satisfying $v(\beta)=1$ and $v(\gamma)\geq 2$.
Thus, by N\'eron's classification it is an $III$ fiber.
\bigskip

Above $1$, we set $s-1=S$ and the Weierstrass model is
$$Y'^2Z'+X'Y'Z'(S^2+S+1)=X'^3-X'^2Z'(S+1)^2(1+S^2)+X'Z'^2S(S^3+1)^2.$$
The change variables $X'=X+SZ$, $Y'=Y+(S+1)X$, $Z'=Z$ gives the model
$$
\begin{array}{l}
Y^2Z+S^2XYZ+(S+S^2+S^3)YZ^2=X^3+X^2Z(S^2-S^4)\\
+XZ^2(S+S^5+S^7)+(-S^3+S^4-S^6+S^8)Z^3
\end{array}
$$
satisfying $v(\gamma)\geq 2$, $v(\beta)=1$.
Thus, by N\'eron's classification, it is an $III$ fiber. 
\bigskip

Above $-1$, we set $s+1=S$ and the Weierstrass model is
$$Y'^2Z'+X'Y'Z'S^2=X'^3+X'^2Z'(S^2+1)-X'Z'^2(S^2-1).$$
The change variables $X'=X+(S^2+1)Z$, $Y'=Y+SX+S^2 Z$, $Z'=Z$ gives the model
$$Y^2Z+(S^2-S)XYZ+S^4 YZ^2=X^3+(1-S^3)X^2Z+(S^4-S^5)XZ^2+S^6 Z^3$$
satisfying $v(\gamma)=6$, $v(\beta)=4\geq 3$, $v(\mu)=4\geq 3$ and $v(\bar{\alpha)}=v(\lambda^2+4\alpha)=0$.
Thus, by N\'eron's classification, it is an $I_6$ fiber. 
\begin{rem}
We may also prove that in characteristic $3$, Beauville's surface has singular fibers $I_6$, $I_3$ and $III$ and use the previous quadratic base change.
\end{rem}

\end{proof}
\begin{rem}
The Galois group $Gal(\bar{\mathbb F_2}|\mathbb F_2)$ operates non trivially on $NS(Y_{10}/\mathbb F_2)$. This result agrees with the extension for $p=2$ (proved by Sch\" {u}tt \cite{Sch2}) of Artin's theorem valid for $p>2$.
\end{rem}

\subsection{The $L$-series }
\bigskip

If  $X$ denotes a $K3$ hypersurface, the $L^*$-series of $X$ is defined by
$$L^*(X,s):=\prod^* Z(X/\mathbb F_p,p^{-s})=\sum_{n\geq 1 } \frac {a(n)}{n^s},$$
where the product is taken over the $p$ of good reduction. Giving a suitable value to the  local factors corresponding to the bad primes, the $L$-series of the surface $X$ can be expressed in terms of the Mellin transform of a modular form.

Let us recall the formula
$$Z(X/\mathbb {F}_p,T)=\frac {1}{(1-T)(1-p^2T)P_2(T)}$$
with $P_2(T)\in \mathbb Z[T]$ of degree $22$ and $P_2(0)=1$.

\begin{lem}
For $Y_{10}$ and $p\neq 2$, $p\neq 3$,

$$P_2(T)=(1-pT)^{17}(1-(\frac {6}{p})pT)^2(1-(\frac {-3}{p})pT)Q_2(T),$$
where $(\frac {.}{p})$ is the  Legendre's symbol.
\end{lem}

\begin{proof}
It follows from lemma 2, if $(\frac {6}{p})=-1$, that the components $\Theta_{\alpha ,i}$ and  $\Theta_{\beta ,i}$, $i=1,2$ are not defined on the base field and the Frobenius exchanges them. Similarly, if $(\frac {-3}{p})=-1$, $\Sigma$ is not defined on $\mathbb F_p$. Therefore, $ \Theta_{\alpha ,i},  \Theta_{\beta ,i}$, $i=1,2$, $\Sigma$ and its conjugate generate a space of dimension $6$ invariant by the Frobenius $ F_p$. In this subspace, we have $\det (1-t F_p)=(1-p^2t^2)^3$. In its orthogonal space of dimension $14$ (for the intersection form) in the N\'eron-Severi lattice, the action of the Frobenius is the multiplication by $p$, since there is a base of cycles defined over the base field.
\end{proof}

If $P_2(T)=(1-pT)^{k_1}(1+pT)^{k_2}Q_2(T)$ with $Q_2(T)=(1-\lambda_1T)(1-\lambda_2T)\in \mathbb Z[T]$, then the number $N_q$ of $\mathbb F_q$-rational points of $X$ for $q=p^r$ is given by the formula
$$N_q=1+q^2+q(k_1+(-1)^rk_2)+A_q$$
with $A_q=\lambda_1^r+\lambda_2^r$.

Now we deduce, from lemma 3, the following corollary. In fact, $A_q$ is the trace of the Frobenius for the corresponding Galois representation. The formula of the corollary is nothing else than the Lefschetz fixed point formula.

\begin{cor} Let $p\neq 2,\,\, 3$.
$$N_q(Y_{10}) =  1+q^2+17q+2q((\frac {6}{p}))^r+q((\frac {-3}{p}))^r+A_q(Y_{10}).$$
\end{cor}

\bigskip

The computation of $N_q$ will finally give $A_q$ and the $L$-series.

\bigskip

\begin{lem}
Let $q=p^r$ and $p\neq 2,3$. We have the following congruences,
 
$$N_q(Y_{10}) \equiv 4q-4+(\frac {3}{p})^r+(\frac {2}{p})^r-2(\frac {6}{p})^r\,\,\,\,\,{\hbox {mod}}\quad 8 $$
$$A_q(Y_{10})  \equiv 3-q^2+3q-q(2(\frac {6}{p})^r+(\frac {-3}{p})^r)+(\frac {3}{p})^r+(\frac {2}{p})^r-2(\frac {6}{p})^r\,\,\,\,\,{\hbox {mod}} \quad 8.$$

\end{lem}

\begin{proof}
Let us define
$$Y:=\{ (x:y:z:t)\in \mathbb P^3 / xyz(x+y+z)+t^2(xy+xz+yz)-10xyzt=0 \}.$$
The $K3$ hypersurface $Y_{10}$ is obtained from $Y$ by blowing up the $7$ double points

$$P_{01}, P_{02}, P_{03}, P_{04}, P_{12}, P_{13}, P_{23}.$$
Let $\mathbb P(Y')$ the projective variety corresponding to 
$$Y':=\{(x:y:z:1)\in \mathbb P^3 /(x:y:z:1)\in Y, xyz\neq 0 \}.$$
First we count the points in $Y_{10}$ with at least one zero coordinate. We get:
\begin{itemize}
\item $q-2$ points with only one zero coordinate of the form $(x:y:z:0)$ with $x+y+z=0$,
\item $6(q-1)$ points with only two zero coordinates (without intersection points) on the double lines $(t=0, \quad x=0)$, $(t=0, \quad y=0)$, $(t=0, \quad z=0)$,\item $3(q-1)$ points with only two zero coordinates (without intersection points) on the lines $(x=0, \quad y=0)$, $(x=0, \quad z=0)$, $(y=0, \quad z=0)$,
\item $4$ points with three zero coordinates $P_{01}$, $P_{02}$, $P_{03}$, $P_{04}$,
\item $3$ points $P_{01}$, $P_{02}$, $P_{03}$, the intersection points of the double lines, 
\item $10q$ points coming from the blowing up of the $7$ double points and $P_{01}$, $P_{02}$ and $P_{03}$.
\end{itemize}
So we find $20q-4$ points with at least one zero coordinate and
$$N_q(Y_{10})=N_q(Y')+20q-4.$$

Now fix $z$, $z\neq 0,\pm 1$ and consider points with $x$ and $y$ $\neq 0, \pm 1$. Denote $N_1$ the number of such points.
If $(x:y:z:1)$ is such a point on the surface, there is also
$$(1/x:y:z:1) \quad (x:1/y:z:1) \quad (x:y:1/z:1) \quad (1/x:1/y:z:1)$$
$$(1/x;y:1/z:1) \quad (x:1/y:1/z:1) \quad (1/x:1/y:1/z:1).$$
Thus $N_1 \equiv 0\qquad (8)$.

Suppose $z=1$. Then $x+1/x+y+1/y=8$.
\begin{itemize}
\item If $x \neq y$ and $x \neq 1/y$, we get the points
$$(x:y:1:1) \quad (1/x:y:1:1) \quad (x:1/y:1:1) \quad (1/x:1/y:1:1)$$
together with their $6$ permuted. So this number is $\equiv 0 \quad (8)$.
\item If $x=y$ or $x=1/y$, then $x+1/x=4$.
So if $(\frac {3}{p})=1$ we get the points
$$(x:x:1:1) \quad (1/x:1/x:1:1) \quad (x:1/x:1:1) $$
and their permuted, that is $12$ points.
\end{itemize}
Suppose $z=-1$. Then $x+1/x+y+1/y=12$.
\begin{itemize}
\item If $x \neq y$ and $x \neq 1/y$, we get the points
$$(x:y:-1:1) \quad (1/x:y:-1:1) \quad (x:1/y:-1:1) \quad (1/x:1/y:-1:1)$$
together with their $6$ permuted. So this number is $\equiv 0 \quad (8)$.
\item If $x=y$ or $x=1/y$, then $x+1/x=6$.
So if $(\frac {2}{p})=1$ we get the points
$$(x:x:-1:1) \quad (1/x:1/x:-1:1) \quad (x:1/x:-1:1) $$
and their permuted that is $12$ points.
\end{itemize}
Now remain the cases

\begin{itemize}
\item (1:y:1:1)  and the permuted i.e. $y+1/y=6$; so if $(\frac {2}{p})=1$ we get $6$ points
\item (-1:y:1:1)  and the permuted i.e. $y+1/y=10$; so if $(\frac {6}{p})=1$ we get $12$ points
\item (-1:y:-1:1)  and the permuted i.e. $y+1/y=14$; so if $(\frac {3}{p})=1$ we get $6$ points
\end{itemize}

Counting all these points, we find
$$N_q(Y')=8t+9((\frac {3}{p})^r+(\frac {2}{p})^r)+6(\frac {6}{p})^r$$
and $$N_q(Y_{10})\equiv 4q-4+(\frac {3}{p})^r+(\frac {2}{p})^r-2(\frac {6}{p})^r\,\,\,\,\,{\hbox {mod}} \quad 8 .$$

Thus by corollary 1
$$A_q(Y_{10})  \equiv 3-q^2+3q-q(2(\frac {6}{p})^r+(\frac {-3}{p})^r))+(\frac {3}{p})^r+(\frac {2}{p})^r-2(\frac {6}{p})^r\,\,\,\,\,{\hbox {mod}} \quad 8.$$

\end{proof}
\begin{lem}
$$L(Y_{10},3)=\frac {1}{2}\sum'_{k,m}\frac {k^2-2m^2}{(k^2+2m^2)^3}=L(f,3).$$
\end{lem}

From Shioda-Inose \cite {SI}, since we know the lattice of transcendental cycles, the zeta function of $Y_{10}$ on $\mathbb F_q$ can be expressed in terms of the zeta function of a model of the elliptic curve $\mathbb C/\mathbb Z+\mathbb Z\sqrt{-2}$.

Computing $N_p$ for small p, we find the corresponding $A_p$ with corollary 1.
\begin{center}
\begin{tabular}{|l||r|r|r|r|r|r|}
\hline

p & 5 & 7 & 11 & 13 & 17 & 19 \\ \hline
$A_p$ & 0 & 0 & 14 & 0 & 2 & -34\\ \hline
\end{tabular}
\end{center}

By inspection, these traces coincide with the first Fourier coefficients of the newform $f=\sum_n a_n q^n$ of level $8$ and weight $3$ from [\cite{Sch1}, Tab. 1]. We shall prove that it holds for every prime.

Following now the method of Stienstra and Beukers \cite {SB}, denoting $A_{p^r}:=\lambda_1^r+\lambda_2^r$, one finds the following possibilities.

$$
\begin{array}{lccl}
{\hbox {if}} &  p=a^2+2b^2 (p \equiv 1, \,\, 3 \,\, {\hbox {mod}}8), &  A_p=\pm 2(a^2-2b^2)  &  \lambda_1 \lambda_2=\pm p^2\\
{\hbox {if}} & p\equiv 5,\,\, 7 \,\, {\hbox {mod}}8,                & {\hbox { either}} \,\,  A_p=0   &  \lambda_1\lambda_2=\pm p^2\\   
             &                                                         & {\hbox {or}}\,\, A_p=\pm 2p, \pm p  &  \lambda_1\lambda_2=p^2.
\end{array}
$$

By lemma 7, if $p\equiv 1 \,\,\,\,{\hbox {or}}\,\,\,3 \,\,\,\, {\hbox {mod}} \quad 8$, then $A_p(Y_{10})\equiv 2p \,\,\,\, {\hbox {mod}} \quad 8$, $A_{p^2} \equiv 2 \,\,\,\, {\hbox {mod}} \quad 8$ hence $A_p(Y_{10})=2(a^2-2b^2)$ and $\lambda_1 \lambda_2 =p^2$.

Furthermore, if  $p\equiv 5 \,\,\,\,{\hbox {or}}\,\,\,7 \,\,\,\, {\hbox {mod}}\quad 8$, then $A_p(Y_{10})\equiv 0 \,\,\,\, {\hbox {mod}} \quad 8$, $A_{p^2} \equiv 2 \,\,\,\, {\hbox {mod}} \quad 8$ hence $A_p(Y_{10})=0$ and $\lambda_1 \lambda_2=-p^2$.

Thus we obtain,
$$
\begin{array}{lccl}
{\hbox {if}} &  p=a^2+2b^2 (p \equiv 1, \,\, 3 \,\, {\hbox {mod}} \quad 8), &  A_p(Y_{10})=2(a^2-2b^2)  &  \lambda_1 \lambda_2= p^2\\
{\hbox {if}} & p\equiv 5,\,\, 7 \,\, {\hbox {mod}} \quad 8,                &   A_p(Y_{10})= 0   &  \lambda_1\lambda_2=-p^2.   
\end{array}
$$

Now, with a good choice of local factors for $p=2$,  $p=3$,  this achieves the proof of theorem 1 1).

\subsection{Proof of theorem 1 2)}

Denote as in \cite{Ber2} 

$$D_{j\tau}=(mj\tau+\kappa)(mj\bar {\tau}+\kappa).$$
Then
$$\begin{aligned}
m(P_k)=\frac {\Im \tau}{8\pi^3}\sum'_{m,\kappa} & [-4\frac
{(m(\tau+\bar {\tau})+2\kappa)^2}{D_{\tau}^3}+\frac {4}{D_{\tau}^2}\\
                                             & +16\frac
{(2m(\tau+\bar {\tau})+2\kappa)^2}{D_{2\tau}^3}-\frac
{16}{D_{2\tau}^2}\\
                                             & -36\frac
{(3m(\tau+\bar {\tau})+2\kappa)^2}{D_{3\tau}^3}+\frac
{36}{D_{3\tau}^2}\\
                                             & +144\frac
{(6m(\tau+\bar {\tau})+2\kappa)^2}{D_{6\tau}^3}-\frac
{144}{D_{6\tau}^2}]
\end{aligned}$$
 For $k=10$, we get $\tau=\frac {-i}{\sqrt {2}}$ and 
$$D_\tau=\frac {1}{2}(m^2+2\kappa^2)$$
$$D_{2\tau}=2m^2+\kappa^2$$
$$D_{3\tau}=\frac {1}{2}(9m^2+2\kappa^2)$$
$$D_{6\tau}=18m^2+\kappa^2.$$
Thus

$$m(P_{10})=\frac {\sqrt {2}}{16\pi^3}[16\times 4\sum'_{m,\kappa}\frac
{m^2-2\kappa ^2}{(m^2+2\kappa ^2)^3}$$
$$-36\times 8 \frac {4k^2}{(9m^2+2k^2)^3}+36 \times 4 \frac {1}{(9m^2+2k^2)^2}$$
$$-36\times 4 \frac {4k^2}{(18m^2+k^2)^3}+36 \times 4 \frac {1}{(18m^2+k^2)^2}].$$

This, in turn, can be written as
$$m(P_{10})=\frac {\sqrt {2}}{16\pi^3}[16\times 4\sum'_{m,\kappa}\frac
{m^2-2\kappa ^2}{(m^2+2\kappa ^2)^3}$$
$$+36\times 8 \frac {9m^2-2\kappa^2}{(9m^2+2\kappa^2)^3}-36 \times 4 \frac {1}{(9m^2+2\kappa^2)^2}$$
$$+36\times 8 \frac {\kappa^2-18m^2}{(\kappa ^2+18m^2)^3}+36 \times 4 \frac {1}{(18m^2+\kappa^2)^2}].$$

\begin{lem}(Zagier \cite{Za})
The following equalities hold:

\begin{enumerate}

 \item$$A(s):= \sum'(-\frac {1}{(9m^2+2k^2)^s}+\frac {1}{(k^2+18m^2)^s}) =2\sum_{n=1}^{\infty} \left ( \frac {-3}{n} \right ) \frac {r_n}{n^s}$$
where 
$$r_n:=\frac {1}{2} \# \{ (k,m) / k^2+2m^2=n \};$$

more precisely
$$A(2)=\frac {\pi^2}{2\sqrt{6}}L(\chi_{-3},2).$$

\item
$$\sum' \frac {k^2-18m^2}{(k^2+18m^2)^s} +\sum' \frac {9m^2-2k^2}{(9m^2+2k^2)^s}=(1+\frac {2}{3^s}+\frac {27}{3^{2s}})\sum'\frac {m^2-2k^2}{(m^2+2k^2)^s}.$$
\end{enumerate}

\end{lem}
\begin{proof}
1) Consider $n=k^2+2m^2$. 

 \begin{itemize}
\item If $3\nmid k$ and $3\mid m$, then $k^2+2m^2=k^2+18m'^2 \equiv 1 \mod 3$.
\item If $3\mid k$ and $3\nmid m$, then $k^2+2m^2=9k'^2+2m^2 \equiv -1 \mod 3$.
\item In the two other cases, $3\mid n$.
\end{itemize}
Thus 
$$A(s)=2\sum_{n=1}^{\infty} \left ( \frac {-3}{n} \right ) \frac {r_n}{n^s}=2L(\chi_{-3},s)L(\chi_{24},s)$$
since
$$\sum_{n=1}^{\infty} \frac {r_n}{n^s}=\zeta (s) L(\chi_{-8},s).$$
If $s=2$, since $2L(\chi_{24},2)=\frac {\pi^2}{2\sqrt{6}}$, we deduce

$$A(2)=\frac {\pi^2}{2\sqrt{6}}L(\chi_{-3},2).$$
2) Denote
$$B(s):=\sum' \frac {k^2-18m^2}{(k^2+18m^2)^s} +\sum' \frac {9m^2-2k^2}{(9m^2+2k^2)^s}$$
and define
$$S:=\frac {1}{2} \sum'_{k,m} \frac {k^2-2m^2}{(k^2+2m^2)^s}.$$
Now 
$$S=\sum_{n=1}^{\infty} \frac {a_n}{n^s}=L(f,s)$$
with $f$ a cusp form of weight $3$ and level $8$.
Moreover, the $L$ function $L(f,s)$ has an Euler product expansion
$$L(s,f)=\frac {1}{1-a_3 3^{-s}+3^{3-1-2s}}\prod_{p\neq 3}L_p(f,s).$$
Since $a_3=-2$, it follows
$$L(s,f)=\frac {1}{1+\frac {2}{ 3^s}+\frac {9}{3^{2s}}}\prod_{p\neq 3}L_p(f,s).$$
Now
$$B(s)=\sum'_{3\mid k, 3\nmid m} \frac {m^2-2k^2}{(m^2+2k^2)^s}+\sum'_{3\mid m, 3\nmid k} \frac {m^2-2k^2}{(m^2+2k^2)^s}+2\sum'_{3\mid k, 3\mid m} \frac {m^2-2k^2}{(m^2+2k^2)^s},$$
$$\sum'_{3\mid k, 3\nmid m} \frac {m^2-2k^2}{(m^2+2k^2)^s}+\sum'_{3\nmid k, 3\mid m} \frac {m^2-2k^2}{(m^2+2k^2)^s}=2(1+\frac {2}{3^s}+\frac {9}{3^{2s}})S$$
and

$$\sum'_{3\mid k, 3\mid m} \frac {m^2-2k^2}{(m^2+2k^2)^s}=9^{1-s}\sum' \frac {m^2-2k^2}{(m^2+2k^2)^s}.$$
So,
$$\begin{aligned}
B(s) &=2S(1+\frac {2}{3^s}+\frac {9}{3^{2s}})+4\times 9^{1-s} S\\
     &=2S(1+\frac {2}{3^s}+\frac {27}{9^s}).
\end{aligned}
$$

\end{proof}

Thus, by the previous lemma, we get
$$m(P_{10})=2d_3+\frac {3\times 8\sqrt{2}}{\pi^3} \sum' \frac {m^2-2\kappa^2}{(m^2+2\kappa^2)^3}$$
$$=2d_3+3m(P_2).$$
This is precisely the relation guessed numerically by Boyd \cite{Bo}
and finally

$$m(P_{10})=2d_3+\frac {1}{9}\frac {\left |\det T_{Y_{10}} \right |^{3/2}}{\pi^3} L(Y_{10},3),$$
that achieves the proof of theorem 1 2).
\end{proof}

\bigskip
\bigskip
\textbf{Acknowledgements}

I thank J. Lewis and N. Yui for their invitation to the CMS Winter meeting in Toronto (December 2006). This was a strong motivation for writing this work and an opportunity for meeting experts in the domain. So, I am particularly indebted to M. Sch\"{u}tt for his thorough reading of the preprint and his precious advice for the correction of some mistakes. My thanks go also to D. Boyd, O. Lecacheux and D. Zagier for their varied and valuable help.

\bigskip
\bigskip

\end{document}